\documentclass{article}
\usepackage{amssymb,amsfonts,amsmath,amsthm,amscd}
\usepackage{makeidx}
\usepackage[english]{babel}
\makeindex
\usepackage{tikz-cd}  
\usepackage{graphicx}
\usepackage{multicol}
\usepackage{tikz}

\usepackage{hyperref}
\input xy
\xyoption{all}
\usepackage[all]{xy}

\usepackage{fancyhdr}
\usepackage{makeidx}
\makeindex 
\usepackage{vmargin} 
\setmargins{3.3cm}       
{3cm}                        
{14.5cm}                      
{23cm}                    
{10pt}                           
{0.5cm}                           
{0pt}                             
{1cm}                           

\numberwithin{equation}{section}

\newtheorem{teo}{Theorem}[section]
\newtheorem{pro}[teo]{Proposition}
\theoremstyle{definition}
\newtheorem{defi}[teo]{Definition}

\newtheorem{lem}[teo]{Lemma}
\newtheorem{ejem}[teo]{Example}

\newtheorem{rem}[teo]{Remark}
\newtheorem{coro}[teo]{Corollary}

\newtheorem{exe}[teo]{Example}
\newcommand{\m}{{}^{-1}}

\newcommand{\Hh}{\mathcal{H}}

\title{\textbf{Partial actions  of groups on hyperspaces}}
\author{
	Luis Mart\'inez, H\'{e}ctor Pinedo and  Edwar Ramirez\\
	\small  Escuela de Matem\'{a}ticas\\
	\small Universidad Industrial de Santander\\
	\small Cra. 27 calle 9, Bucaramanga, Colombia\\
	\small  e-mail: luchomartinez9816@hotmail.com, hpinedot@uis.edu.co, edwar5119@gmail.com\\
}

\date{\today}

\begin{document}

	\maketitle
	\begin{abstract} Let $X$ be a compact Hausdorff space. In this work  we translate partial actions of $X$ to partial actions on some hyperspaces determined by $X,$ this gives an endofunctor  $2^{-}$  in the category of partial actions on compact Hausdorff spaces which generates a monad in this category. Moreover, structural relations between partial actions $\theta$ on $X$ and partial determined by $2^{\theta}$ are established.
		
	\end{abstract}

	\section{Introduction}  Given   an action $\mu: G\times Y\to Y$ of a group $G$ on a set   $Y$ and an invariant subset $X$ of $Y$ (i.e., $\mu(g, x) \in X,$ for all $x \in X$, and  $g\in G$), the restriction of $\mu$ to $G \times X$ determines an action of $G$ on $X.$ However, if $X$ is not invariant, we obtain a partial action on $X$. This is a collection of  partially defined maps  $\theta_g\ (g\in G)$ on $X$ satisfying $\theta_1 = {\rm id}_X$ and $\theta_{gh}$ is an extension of the composition  $\theta_g\circ \theta_h,$ for all $g,h\in G.$  The notion of partial action of a group was introduced by R. Exel in \cite{Ruy1, R} motivated by problems arising from $C^*$-algebras. Since then partial group actions have appeared in many different context, such as the theory of operator algebras, algebra, the theory of R-trees, tilings and model theory (see for instance \cite{KL}).
In the topology, partial actions on topological spaces consist of a family of homeomorphism between open subsets of the space, and have been considered in the context of Polish spaces (see \cite{PU,PU2}), $2$-cell complexes (see \cite{St}), topological semigroups \cite{Choi} and recently in \cite{MPV} where introduced in the realm of profinite spaces.  

It seems that when a partial action on some structure is given, one of the most relevant
problems is the question of the existence and uniqueness of a globalization, that is, if whether a partial action can be realized as restrictions of a corresponding collection of total maps on some superspace.  In the topological context,  this problem was studied by Abadie \cite{AB}  and independently by Kellendonk and Lawson \cite{KL}. It was proved that for any continuous partial action $\theta$ of a topological group $G$ on a topological space $X,$ there is a topological space $Y$ and a continuous action $\mu$ of $G$ on $Y$ such that $X$ is a subspace of $Y$ and $\theta$ is the restriction of $\mu$ to $X.$ Such a space $Y$ is called a globalization of $X.$ They also show that there is a minimal globalization $X_G$ called the enveloping space of $X$ (see subsection \ref{globa} for details). Recent  topological advances on partial actions on (locally) compact spaces  include the groupoid approach to the enveloping spaces associated to partial actions of countable discrete groups \cite{EGG}. Also several classes of $C^*$-algebras can be described as partial crossed products that correspond to partial actions of discrete groups on profinite spaces; for instance the Carlsen-Matsumoto $C^*$ -algebra $\mathcal{O}_X$ of an arbitrary subshift $X$ (see \cite{DOE}). The interested reader may consult \cite{D} and  \cite{Ruy} for a detailed account in  developments around partial actions.

On the other hand, the study of hyperspaces has developed for more than one hundred years, topological properties in hyperspaces: dimension, shape, contractibility, admissibility, unicoherence, etc., have been topics where researchers have dedicated a lot of attention recently. Furthermore, there are many papers in different areas of mathematics focused on the study of set-valued function where hyperspaces are the natural environment to work. Also, it is interesting to study when a hyperspace can be embedded in a cell or when a cell can be embedded in a hyperspace, topics concerning  the $n$-od problem,  Whitney properties and Whitney-reversible properties have been widely considered ( see for instance \cite{NW} and the reference therein).

This work is structured as follows. After the introduction in  Section \ref{pactions} we present the preliminary notions on topological partial actions and their enveloping actions, at the end of this section we fix a compact Hausdorff space $X$ and state our conventions, notations and results  on the  hyperspaces $\Hh_1$, $\Hh_2$ and $\Hh_3$ consisting of compact, compact and connected, and finite subsets of $X$, respectively.  In Section \ref{XtoH} we  translate partial actions  $\theta$   of $X$ to partial actions $2^\theta$ on $\Hh\in \{\Hh_1, \Hh_2, \Hh_3\}$ and present in Theorem \ref{teo2.8} and Proposition \ref{cer} some structural properties preserved by this correspondence.  Separation properties relating enveloping actions of $\theta$ and $2^\theta$ are considered in Corollary \ref{separ2}  and Theorem \ref{regu}. Finally,  Section \ref{cate} has a categorical flavor, where it is considered  the category \textbf{$G\curvearrowright$ CH} whose objects are  topological partial actions on compact Hausdorff spaces and show in Theorem \ref{monad} that the functor $2^-$ generates a monad in this category.

\section{The notions}\label{pactions}
We present the necessary background on partial actions and hyperspaces that we use throughout the work. 	
	\subsection{Preliminaries on partial actions and their enveloping actions} We start with  the following.

\begin{defi} \cite[p. 87-88]{KL} Let $G$ be a group  with identity element $1$ and $X$ be a  set. A partially defined  function $\theta: G\times X\dashrightarrow X$, $(g,x)\mapsto g\cdot x$ is called a (set theoretic)  {\it  partial action} of $G$ on $X$ if for each $g,h\in G$ and $x\in X$ the following assertions hold:
	\begin{enumerate}
		\item [(PA1)] If $\exists g\cdot x$, then $\exists g^{-1}\cdot (g\cdot x)$ and $g^{-1}\cdot (g\cdot x)=x$,
		\item [(PA2)] If $\exists g\cdot(h\cdot x)$, then $\exists (gh)\cdot x$ and $g\cdot(h\cdot x)=(gh)\cdot x$,
		\item [(PA3)] $\exists 1\cdot x$ and $1\cdot x=x,$
	\end{enumerate}
where $\exists g\cdot x$ means that $g\cdot x$ is defined. We say that $\theta$ {\it acts} (globally)  on $X$  or that $\theta$ is global if $\exists g\cdot x,$ for all $(g,x)\in G\times X.$
\end{defi}

Given a partial action  $\theta$  of $G$ on $X,$  $g\in G$ and $x\in X$. We set:
\begin{itemize} 
\item   $G*X=\{(g,x)\in G\times X\mid \exists g\cdot x\}$  the domain of $\theta.$
\item  $X_g=\{x\in X\mid \exists g\m\cdot x \}.$

\end{itemize}	
Then $\theta$ induces a family of bijections $(\theta_g\colon X_{g\m}\ni x\mapsto g\cdot x\in X_g \}_{g\in G},$  
We also denote this family  by $\theta.$
The following result characterizes partial actions in terms of a family of bijections.
\begin{pro}\label{prop2.3} {\rm  \cite[Lemma 1.2]{QR} \label{fam} A partial action $\theta$ of $G$ on $X$ is a family $\theta=\{\theta_g\colon X_{g\m}\to X_g\}_{g\in G},$ where $X_g\subseteq X,$  $\theta_g\colon X_{g\m}\to X_g$ is bijective, for  all $g\in G,$  and such that:
\begin{itemize}
\item[(i)]$X_1=X$ and $\theta_1=\rm{id}_X;$
\item[(ii)]  $\theta_g( X_{g\m}\cap X_h)=X_g\cap X_{gh};$
\item[(iii)] $\theta_g\theta_h\colon X_{h\m}\cap  X_{ h\m g\m}\to X_g\cap X_{gh},$ and $\theta_g\theta_h=\theta_{gh}$ in $ X_{h\m}\cap  X_{ h\m g\m};$
\end{itemize}
for all $g,h\in G.$}
\end{pro}
 In view of Proposition \ref{prop2.3} a partial action on $X$ are frequently denoted as a family of maps  $(\theta_g, X_g)_{g\in G},$ between subsets of $X$ satisfying conditions (i)-(iii) above.

From now on in this work $G$ will denote a topological group and $X$ a topological space. We endow $G\times X$ with the product topology and $G*X$ with the topology of subspace. Moreover $\theta: G*X\to X$  will denote a partial action.  We say that $\theta$ is a {\it topological partial action} if $X_g$ is open and $\theta_g$ is a homeomorphism, for all $g\in G$. Moreover, if $\theta$ is continuous, $\theta$ 
is  called a {\it continuous partial action}.


	\subsection{Restriction of global  actions and globalization}\label{globa}	

Let $\mu \colon G\times Y\to Y$ be a continuous  action of $G$ on  a  topological space $Y$ and $X\subseteq Y$ be an open set.  Then we can obtain by restriction a  topological partial action on $X$ by setting:
\begin{equation*}\label{induced}X_g=X\cap \mu_g(X),\,\,\, \,\, \theta_g=\mu_g\restriction X_{g\m}\,\,\,  \text{ and }\, \,\,\,\theta \colon G* X\ni (g,x)\mapsto \theta_g(x)\in X .\end{equation*}  Then $\theta$ is a topological partial action of $G$ on $X,$ we say that   $\theta$   is  the {\it restriction of $\mu$ to X}. 


As mentioned in the introduction, a  natural problem in the study of partial actions is whether they can be restriction of global actions. In the topological sense, this turns out to be affirmative and a proof was given in \cite[Theorem 1.1]{AB} and independently in \cite[Section 3.1]{KL}.  Their construction is as follows. Let $\theta$ be a topological partial action of $G$ on $X$ and consider the following equivalence relation on $ G\times X$:
\begin{equation}
\label{eqgl}
(g,x)R(h,y) \Longleftrightarrow x\in X_{g\m h}\,\,\,\, \text{and}\, \,\,\, \theta_{h\m g}(x)=y.
\end{equation}
 Denote by  $[g,x]$   the equivalence class of  $(g,x).$
The {\it enveloping space}  or the {\it globalization}  of $X$ is the set  $X_G=(G\times X)/R$  endowed the quotient topology.  We have by  \cite[Theorem 1.1]{AB}  that the action 
\begin{equation}
\label{action}
\mu \colon G\times X_G\ni (g,[h,x])\to [gh,x]\in X_G,
\end{equation}
is continuous and is the so called the {\it enveloping action} of $\theta.$  Further by (ii) in  \cite[Proposition 3.9]{KL}  the map 
\begin{equation} \label{op}q: G\times X\ni (g,x)\mapsto [g,x]\in X_G,
\end{equation}
is open. Moreover
\begin{equation}
\label{iota}
\iota \colon X\ni x\mapsto [1,x]\in X_G
\end{equation}  satisfies $G\cdot \iota(X)=X_G$ and   it follows by \cite[Proposition 3.12]{KL} that  is a homeomorphism onto $\iota(X)$ if and only if   $\theta$ is continuous, and by \cite[Proposition 3.11]{KL}  $\iota(X)$ is open in $X_G,$ provided that $G*X$ is open.

We finish this section with a result that will be useful in the sequel.

\begin{lem}\label{separ}
	Let  $\mu: G\times Y\rightarrow Y$ be a continuous  global action of topological group on a topological space $Y$ and let $U\subseteq Y$  such that  $G\cdot U=Y$. Then the following assertions hold.
\begin{enumerate}
\item [(i)] If  $G$  and $U$ are separable, then $Y$ is separable.
\item [(ii)]If $U$ is closed and regular, then $Y$ is regular.
\end{enumerate}
\end{lem}
\begin{proof}
	(i) Let  $\{u_n: n\in \mathbb{N}\}\subseteq U$ and  $\{g_m: m\in \mathbb{N}\}\subseteq G$ denses subsets of $U$ and $G$, respectively.  Then for an open nonempty set  $V\subseteq Y$ we have that $W:=\mu^{-1}(V)\cap (G\times U)$ is open in $G\times U$.  Then there are  $n,m\in \mathbb{N}$ such that $(g_m,u_n)\in W$ and consequently, $g_m\cdot u_n\in V$ which implies that  $\{g_m\cdot u_n\in Y: m,n\in\mathbb{N}\}$ is dense in $Y$.

(ii)  	Take $y\in Y$ and $Z\subseteq Y$ an open set such that $y\in Z.$ The fact that $G\cdot U=Y$ implies that  there are  $g\in G$, $u\in U$ such that $y=g\cdot u$. Since  $\mu$ is continuous there is an open set $B\subseteq Y$ for which $u\in B$ and  $g\cdot B\subseteq Z$. Then  $V=U\cap B$ is open in $U$ and $g\cdot V\subseteq Z$. Since $U$ is regular, there is an open set $W$ of $U$ such that $u\in W\subseteq Cl_U(W)\subseteq V$ but  $U$ is closed then  \begin{center}
		$y=g\cdot u\in g\cdot W\subseteq g\cdot Cl_U(W)=g\cdot \overline{W}=\overline{g\cdot W}\subseteq g\cdot V\subseteq Z$,
	\end{center}
and $Y$ is regular. 
\end{proof}

\subsection{Conventions on hyperspaces}
{\it From now on in this work $X$ will denote a compact Hausdorff space.} 

The  hyperspace $\mathcal{H}_1:=2^X$ is the set  consisting of non-empty compact subsets of   $X$. For $U_1, U_2,\cdots, U_n$  non-empty open sets of  $X$, let 
$$\langle U_1,...,U_n\rangle_{\mathcal{H}_1}=\left\{A\in \mathcal{H}_1:A\subseteq\bigcup\limits_{i=1}^nU_i\text{ y }A\cap U_i\neq\emptyset,\ 1\leq i\leq n\right\},$$
moreover we set  $\langle \emptyset\rangle:=\emptyset$. 
The {\it vietoris topology} on $\mathcal{H}_1$ is the one generated by open collections of the form  $\langle U_1,...,U_n\rangle.$ 
We shall also work with the subspaces 
$$\mathcal{H}_2:=\{C\in \Hh_1\mid C \text{\,\,is\,\, connected}\}\,\,\,\text{ and}\,\,\,\mathcal{H}_3:=\{F\in \Hh_1\mid F  \rm{\,\,is\,\,finite}\}$$ that is $\langle U_1,\cdots,U_n\rangle_{\mathcal{H}_i}:=\mathcal{H}_i\cap \langle U_1,\cdots,U_n\rangle_{\mathcal{H}_1},$  for $U_1, U_2,\cdots, U_n$ open subsets of $X$ and $i=2,3$. Finally, when taking about a hyperspace $\mathcal{H}$ we make reference to any of the spaces $\mathcal{H}_1, \mathcal{H}_2$ as well as $\Hh_3$.

We summarize some well-known  properties of the space $ \mathcal{H}.$ For more details on hyperspaces, the interested reader may consult \cite{NW}.
\begin{lem}\label{proph} Let $X$ be a compact Hausdorff space. Then the following assertions hold.
 \begin{itemize}
\item [(i)] The map $X\ni x\mapsto \{x\}\in \mathcal{H}$  is an embedding of $X$ into $\Hh.$ 
\item [(ii)] $\mathcal{H}$ is a compact Hausdorff space and the map $u:2^{2^X}\rightarrow 2^X$, $A\mapsto \cup A$ is continuous.
\end{itemize}
\end{lem}

\section{From partial actions on $X$ to partial actions on $\Hh$} \label{XtoH} 
In what follows we shall use a continuous partial action on $X$ to construct a continuous partial action on $ \mathcal{H}.$ We start with the next.

\begin{lem}\label{lem1}
	Let  $U$ and $V$ be open subsets of  $X$ and $f:U\rightarrow V$ a homeomorphism, then the map $2^f: \langle U\rangle_{\mathcal{H}} \ni A \mapsto f(A)\in  \langle V\rangle_{\mathcal{H}}$
	is a  homeomorphism.
\end{lem}
\begin{proof} We shall prove the case $\mathcal{H}=\mathcal{H}_2$ the cases $\Hh=\Hh_1$ and $\Hh_3$ are similar.   First of all notice that 
	$2^f$ is well defined since  $f$ is continuous. We check that it is continuous.
	Take $A\in \langle U\rangle_{\mathcal{H}_2}$ and $V_1,\cdots,V_k$ open subsets of $X$ such  that  $$f(A)\in \langle V\rangle_{\mathcal{H}_2} \cap \langle V_1,\cdots,V_k\rangle_{\mathcal{H}_1}=\mathcal{H}_2\cap\langle V\rangle_{\mathcal{H}_1}\cap \langle V_1,\cdots,V_k\rangle_{\mathcal{H}_1} =\langle V\cap V_1,\cdots,V\cap V_k\rangle_{\mathcal{H}_2}.$$ 
	Consider the open subset of  $\langle U\rangle_{\mathcal{H}_2}$ defined by $Z:=\langle U\rangle_{\mathcal{H}_2} \cap \langle f^{-1}(V_1\cap V),\cdots,f^{-1}(V_k\cap V)\rangle_{\mathcal{H}_1}.$
	Note that  $A\in Z$ because  $f(A)\subseteq \bigcup\limits_{i=1}^k (V\cap V_i)$, thus $A\subseteq \bigcup\limits_{i=1}^kf^{-1}(V\cap V_i).$ Moreover, since $f(A)\cap (V\cap V_i)\neq\emptyset$, we see that  $A\cap f^{-1}(V\cap V_i)\neq \emptyset$, for each $i=1,\cdots,k$.  To check that $2^f(Z)\subseteq \langle V\rangle_{\mathcal{H}_2} \cap \langle V_1,\cdots,V_k\rangle_{\mathcal{H}_1}$, take  $W\in Z$. Since  $W\subseteq \bigcup\limits_{i=1}^k f^{-1}(V_i\cap V)$, we get $2^f(W)\subseteq \bigcup\limits_{i=1}^k (V_i\cap V)$. Also  $W\cap f^{-1}(V_i\cap V)\neq\emptyset$ gives  $2^f(W)\cap (V_i\cap V)\neq\emptyset$ for each $i=1,\cdots,k$ and we conclude that   $2^f$ is continuous. Finally, since $(2^f)^{-1}=2^{f^{-1}}$, the map $2^f$ is a homeomorphism, as desired. 
\end{proof}

\begin{teo}\label{teo2.8}
	Let $\theta:=(\theta_g, X_g)_{g\in G}$ be a  topological partial  action of $G$ on $X$. For  $g\in G$, we set  $2^{\theta_g}:\langle X_{g^{-1}}\rangle_{\mathcal{H}}\ni A \mapsto \theta_g(A)\in  \langle X_g \rangle_{\mathcal{H}}.$ Then $2^{\theta}=(2^{\theta_g}, \langle X_g \rangle_{\mathcal{H}})_{g\in G}$ is a topological partial  action of $G$ on $\mathcal{H}$ and the following assertions hold. 
\begin{enumerate}
\item [(i)] $G*\Hh$ is open provided that $G*X$ is open.
\item [(ii)] If $\theta$ is continuous, then $2^\theta$ is continuous.
\item [(iii)] If $\theta$ is global then $2^\theta$ is global.
\end{enumerate}
\end{teo}
\begin{proof}
	As in the proof of Lemma \ref{lem1} we shall only deal with the case  ${\mathcal{H}}=\mathcal{H}_2.$ By Lemma \ref{lem1} we have that  $2^{\theta_g}$ is a  homeomorphism between open subsets of  $\Hh_2,$ for any  $g\in G$.  We shall check that conditions  (i)-(iii) in Proposition \ref{prop2.3} hold.  To see (i)  notice that $2^{\theta_e}$ is the identity map of  $\langle X\rangle_{\mathcal{H}_2}=\mathcal{H}_2$.  For (ii) take  $g,h\in G$ and  $A\in\langle X_{g^{-1}}\rangle_{\mathcal{H}_2} \cap \langle X_h\rangle_{\mathcal{H}_2}=\langle X_{g^{-1}}\cap X_h\rangle_{\mathcal{H}_2}$, then 
		$2^{\theta_g}(A)=\theta_g(A)\subseteq \theta_g(X_{g^{-1}}\cap X_h)\subseteq X_{gh}$, and thus  $2^{\theta_g}(\langle X_{g^{-1}}\rangle_{\mathcal{H}_2}\cap \langle X_h\rangle_{\mathcal{H}_2})\subseteq \langle X_{gh}\rangle_{\mathcal{H}_2}$. For (iii) take  $A\in \langle X_{h^{-1}}\rangle_{\mathcal{H}_2} \cap \langle X_{(gh)^{-1}}\rangle_{\mathcal{H}_2}=\langle X_{h^{-1}}\cap X_{(gh)^{-1}}\rangle_{\mathcal{H}_2}$, then  \begin{center}
		$2^{\theta_{gh}}(A)=\theta_{gh}(A)=\theta_g(\theta_h(A))=2^{\theta_g}(2^{\theta_h}(A)),$
	\end{center}
and we conclude that $2^{\theta}$ is a partial action of $G$ on $\Hh_2.$ Now we check $(i)-(iii)$.

$(i)$ Suppose that $G*X$ is open in $G\times X$. To see that    $G*\mathcal{H}_2$ is open in  $G\times \mathcal{H}_2$, take $(g,A)\in G*\mathcal{H}_2$. Since $A\subseteq X_{g^{-1}}$, we have  $(g,a)\in G*X$ for all $a\in A$. Now the fact that  $G*X$ is an open subset of $G\times X$, implies that for any $a\in A$ there are open sets  $U_a\subseteq G$ and $V_a\subseteq X$ for which  $(g,a)\in U_a\times V_a\subseteq G*X$. Since  $A$ is compact, there exist $a_1,\cdots,a_n\in A$ with  $A\subseteq \bigcup\limits_{i=1}^n V_{a_i}$, and $A\in \langle V_{a_1},\cdots,V_{a_n}\rangle_{\Hh_2}$. Let  $U:=\bigcap\limits_{i=1}^nU_{a_i}$, then $(g,A)\in U\times \langle V_{a_1},\cdots,V_{a_n}\rangle_{\Hh_2}$ we claim that  
		$ U\times \langle V_{a_1},\cdots,V_{a_n}\rangle_{\Hh_2}\subseteq G*\mathcal{H}_2$.    
Indeed, take $(h,B)\in  U\times \langle V_{a_1},\cdots,V_{a_n}\rangle_{\Hh_2},$ we shall check  $B\in \langle X_{h^{-1}}\rangle_{\Hh_2}$. Take  $b\in B$. Since   $B\subseteq \bigcup\limits_{i=1}^nV_{a_i}$,  there is  $1\leq i\leq n$ for which $b\in V_{a_i}$ and  $(h,b)\in U_{a_i}\times V_{a_i}\subseteq G*X$, then  $b\in X_{h^{-1}}$. From this we get  $B\in \langle X_{h^{-1}}\rangle_{\Hh_2}$ and thus  $(h,B)\in G*\mathcal{H}_2$. This shows that  $G*\mathcal{H}_2$ is open in  $G\times \mathcal{H}_2$.
	
	$(ii)$ Suppose that $\theta$ is continuous. We need to show that 
		$2^{\theta}: G*\mathcal{H}_2\rightarrow \mathcal{H}_2$, $(g,A)\mapsto \theta_g(A)$ is continuous. Let  $(g,A)\in G*\mathcal{H}_2$ and take $V_1,\cdots,V_k$ open subsets of $X$ such that $\theta_g(A)\in \langle V_1,\cdots,V_k\rangle_{\Hh_2}$. For each $a\in A$ there is  $1\leq i_a\leq k$ such that  $\theta_g(a)\in V_{i_a}$, and since  $\theta$  is continuous there are open sets  $U_{i_a}\subseteq G$ and $W_{i_a}\subseteq X$ such that: \begin{center}
		$(g,a)\in (U_{i_a}\times W_{i_a})\cap (G*X)$ and  $\theta((U_{i_a}\times W_{i_a})\cap G*X)\subseteq V_{i_a}$.    
	\end{center}
	The fact that  $A$ is compact implies that there are  $a_1,\cdots,a_m\in A$ such that  $A\subseteq \bigcup\limits_{j=1}^m W_{i_{a_j}}$. 
	
	On the other hand, since  $\theta_g(A)\cap V_j\neq\emptyset$, for any  $1\leq j\leq k$, there are  $r_1,r_2,\cdots,r_k\in A$ for which $\theta_g(r_j)\in V_j$, for all  $j=1,2,\cdots,k$. Set  $i_{r_j}:=j$, $j=1,\cdots,k$. Without loss of generality we may suppose  $\{r_1,r_2,\cdots,r_k\}\subseteq \{a_1,\cdots,a_m\}$ and  $r_i=a_i$, for each $i=1,\cdots,k$. Let  $U:=\bigcap_{j=1}^m U_{i_{a_j}}$. Then  $(g,A)\in Z:=(U\times \langle W_{i_{a_1}},\cdots,W_{i_{a_m}}\rangle_{\Hh_2})\cap (G*\mathcal{H}_2).$ To finish the proof it is enough to show that  $2^{\theta}(Z)\subseteq \langle V_1,\cdots,V_k\rangle_{\Hh_2}$. For this take $(h,B)\in Z$ and  $b\in B$. Since   $B\subseteq \bigcup\limits_{j=1}^m W_{i_{a_j}}$, there exists $1\leq j\leq m$ such that $b\in W_{i_{a_j}}$ and thus $(h,b)\in U_{i_{a_j}}\times W_{i_{a_j}}$. But $B\subseteq X_{h^{-1}}$, then  $(h,b)\in (U_{i_{a_j}}\times W_{i_{a_j}})\cap (G*X) $ and  $\theta_h(b)\in V_{i_{a_j}}$ which implies   $\theta_h(B)\subseteq \bigcup\limits_{i=1}^kV_i$. Finally,  for  $1\leq l\leq k$ we see that  $\theta_h(B)\cap V_l\neq\emptyset$.  Indeed, take  $b\in B\cap W_{i_{a_l}}$, where $a_l=r_l$.  Since $h\in U_{i_{a_l}}$, we have  $(h,b)\in (U_{i_{a_l}}\times W_{i_{a_l}})\cap (G*X)$ and $\theta_h(b)\in V_{i_{a_l}}\cap \theta_h(B)=V_{i_{r_l}}\cap\theta_h(B)=V_l\cap \theta_h(B)$ which finishes the proof of the second item.  

$(iii)$ This is clear.
\end{proof}
\begin{rem} Given a partial action $\theta$ of $G$ on $X,$ we shall refer to $2^\theta$ as the induced partial action of $\theta$  on $\mathcal{H}.$
\end{rem}

\begin{ejem}
	There is a  topological partial action of $\mathbb{Z}_4$ on $S^1$ given by  the family  $\{X_n\}_{n\in \mathbb{Z}_4}$ as it is shown below. 
	
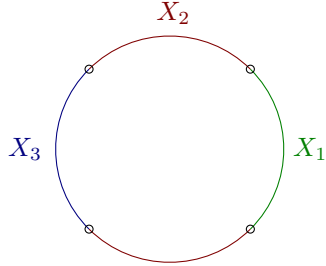
\begin{center}
		\begin{tikzpicture}[scale=1.5]
		\draw[green!50!black] (0.7071,-0.7071) arc (-45:45:1) (1,0) node[anchor=west]{$X_1$};
		\draw[blue!50!black] (-0.7071,0.7071) arc (135:225:1) (-1.5,0) node[anchor=west]{$X_3$};
		\draw[red!50!black] (0.7071,0.7071) arc (45:135:1)  (-0.7071,-0.7071) arc (225:315:1) (-0.2,1.2) node[anchor=west]{$X_2$};
		\draw (0.7071,-0.7071) circle (1pt) (0.7071,0.7071) circle (1pt) (-0.7071,0.7071) circle (1pt) (-0.7071,-0.7071) circle (1pt);
		\end{tikzpicture}
	\end{center}

	\begin{itemize}
		\item[] $\theta_0={\rm Id}_{\mathcal  S^1}$; $\theta_1:X_3\to X_1$ by $\theta_1(e^{it})=e^{i(t+\pi)}$;  $\theta_3=\theta_1^{-1}$, $\theta_2:X_2\to X_2$ is the identity.
	\end{itemize}
	We construct the induced partial action of  $\mathbb{Z}_4$  on  $C(\mathcal S^1)$, for this we  find a  homeomorphism $h$ between  $C(\mathcal S^1),$ the connected sets of $S^1$  and  $D=\{z\in \mathbb{C}:|z|\leq 1\}.$
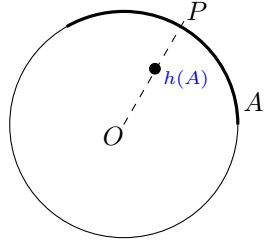
	\begin{multicols}{2}
		\begin{center}
			\begin{tikzpicture}
			\draw[very thick]  (1.5,0) arc (0:120:1.5);
			\draw (0,0) circle (1.5)  (-0.4,-0.15) node[anchor=west]{$O$} (0.4,0.6) node[anchor=west]{\scriptsize{\textcolor{blue}{$h(A)$}}} (0.7,1.5) node[anchor=west]{$P$} (1.45,0.3) node[anchor=west]{$A$};
			\draw[dashed] (0,0) -- (0.8,1.385);
			\filldraw (0.41,0.74) circle (2pt);
			\end{tikzpicture}
		\end{center}
	
\small{Let  $P\in \mathcal S^1$ and take an arc center at  $P$ of length \textit{l}, this arc is mapped  on  $h(A)=\left(1-\frac{l}{2\pi}\right)P\in D$. The arc $\{P\}$ of length zero is mapped onto \textcolor{blue}{$h(\{P\})=P\in D$}}
\normalsize
	\end{multicols}
	\noindent In particular, all arcs  centered at  $P$ are mapped on $\overline{OP}$. From this follows that the sets   $\{\langle X_n\rangle_{\Hh_2}\}_{n\in \mathbb{Z}_4}$ are 
	
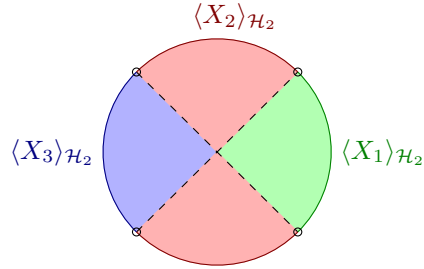
\begin{center}
		\begin{tikzpicture}[scale=1.5]
		\filldraw [green!30] (-0.7071,-0.7071)--(0,0)--(0.7071,-0.7071) arc (-45:45:1);
		\filldraw [blue!30] (-0.7071,-0.7071)--(0,0)--(-0.7071,0.7071) arc (135:225:1);
		\filldraw [red!30] (-0.7071,0.7071)--(0,0)--(0.7071,0.7071) arc (45:135:1)
		(0.7071,-0.7071)--(0,0)--
		  (-0.7071,-0.7071) arc (225:315:1);
		
		\draw [dashed] (-0.7071,-0.7071)--(0.7071,0.7071) (-0.7071,0.7071)--(0.7071,-0.7071) ;
		\draw[green!50!black] (0.7071,-0.7071) arc (-45:45:1) (1,0) node[anchor=west]{$\langle X_1\rangle_{\Hh_2}$};
		\draw[blue!50!black] (-0.7071,0.7071) arc (135:225:1) (-1.9,0) node[anchor=west]{$\langle X_3\rangle_{\Hh_2}$};
		\draw[red!50!black] (0.7071,0.7071) arc (45:135:1)  (-0.7071,-0.7071) arc (225:315:1) (-0.3,1.2) node[anchor=west]{$\langle X_2\rangle_{\Hh_2}$};
		\draw (0.7071,-0.7071) circle (1pt) (0.7071,0.7071) circle (1pt) (-0.7071,0.7071) circle (1pt) (-0.7071,-0.7071) circle (1pt);
		\end{tikzpicture}
	\end{center}
We construct $2^\theta.$ The map	$2^{\theta_2}$ is the identity on $\langle X_2\rangle_{\Hh_2}$.  Notice that  $2^{\theta_1}$ rotates each arc  in $X_3$ $\pi$ radians to an arc  in $X_1$ of the same length. 
	\begin{multicols}{2}
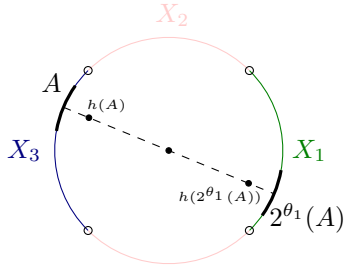
		\begin{center}
			\begin{tikzpicture}
			[scale=1.5]
			\draw[green!50!black] (0.7071,-0.7071) arc (-45:45:1) (1,0) node[anchor=west]{$X_1$};
			\draw[blue!50!black] (-0.7071,0.7071) arc (135:225:1) (-1.5,0) node[anchor=west]{$X_3$};
			\draw[red!20] (0.7071,0.7071) arc (45:135:1)  (-0.7071,-0.7071) arc (225:315:1) (-0.2,1.2) node[anchor=west]{$X_2$};
			\draw (0.7071,-0.7071) circle (1pt) (0.7071,0.7071) circle (1pt) (-0.7071,0.7071) circle (1pt) (-0.7071,-0.7071) circle (1pt);
			\draw[very thick]  (-0.819,0.573) arc (145:170:1) (0.8191,-0.5735) arc (325:350:1);
			\draw[dashed] (-0.9238,0.3826)--(0.9238,-0.3826);
			\filldraw (0,0) circle (0.7 pt) (-0.7,0.29) circle (0.7 pt) (0.7,-0.29) circle (0.7 pt);
			\draw (-1.2,0.6) node[anchor=west]{$A$} (0.8,-0.6) node[anchor=west]{$2^{\theta_1}(A)$} (-0.8,0.4) node[anchor=west]{\tiny{$h(A)$}} (0,-0.4) node[anchor=west]{\tiny{$h(2^{\theta_{\tiny{1}}}(A))$}};
			\end{tikzpicture}
		\end{center}
Then $h(2^{\theta_1}(A))$ is obtained by rotating  $h(A)$ $\pi$ radians, from this  $2^{\theta_1}$ in $D$ is identified with $2^{\theta_1}:\langle X_3\rangle_{\Hh_2}\longrightarrow\langle X_1\rangle_{\Hh_2}$, $re^{it}\longmapsto re^{i(t+\pi)},$ analogously $2^{\theta_3}:\langle X_1\rangle_{\Hh_2}\longrightarrow\langle X_3\rangle_{\Hh_2}$, $re^{it}\longmapsto re^{i(t+\pi)}$.
	\end{multicols}
\end{ejem}

We finish this section with the next.

\begin{pro}\label{cer}
	Let $\theta$ be a topological partial action of $G$ on $X$. If $G*X$ is closed, then   $G*\Hh$ is closed.
\end{pro}
\begin{proof}
	Take $(g,A)\in (G\times \mathcal{H})\setminus G*\mathcal{H}$. Then there is  $a\in A$ such that  $\nexists g\cdot a$ and $(g,a)\notin G*X$ and there are open sets  $U\subseteq G$ and $V\subseteq X$ such that $(g,a)\in U\times V \subseteq (G\times X)\setminus G*X$. Note that $(g,A)\in U\times \langle V,X\rangle_{\mathcal{H}}$  to finish the proof we need to show that  $U\times \langle V,X\rangle_{\mathcal{H}}\subseteq (G\times \mathcal{H})\setminus G*\mathcal{H}$. Take  $(h,B)\in U\times\langle V,X\rangle_{\mathcal{H}}$ and  $b\in B\cap V$. Since  $(h,b)\in U\times V$, we get  $\nexists h\cdot b$, then $\nexists h\cdot B$ and  $(h,B)\notin G*\mathcal{H}$ as desired.
\end{proof}

\subsection{Separation properties and enveloping spaces }\label{apglob}
It is shown in \cite[Proposition 1.2]{AB} that a partial action has a Hausdorff enveloping space if and only if the graph of the action is closed. Below we show  that partial actions on  compact Hausdorff spaces have Hausdorff enveloping space, 
 if and only if  the enveloping space of the induced partial action on  $\mathcal{H}$ is Hausdorff.

From now on, $2^R$ denotes the equivalence relation associated to the enveloping action of the partial action $2^\theta$ of $G$ on $\mathcal{H}$ (see equation \eqref{eqgl}). That is $\Hh_G=(G\times \Hh)/2^R.$ 
\begin{lem}\label{l1}
	Let  $\theta$ be a partial action on $X$ and  $2^{\theta}$ be the corresponding partial action of $G$ on $\Hh$, then the map $\Theta: X_G\ni [g,x]\mapsto [g,\{x\}]\in \mathcal{H}_G$ is an embedding.
\end{lem}
\begin{proof}
	First of all observe that  $\Theta$ is well defined. Indeed,  if $(g,x) R (h,y),$ then $\{x\}\subseteq \langle X_{g^{-1}h}\rangle_{\Hh}$ and $2^{\theta_{h^{-1}g}}(\{x\})=\{y\}$,  which gives  $(g,\{x\})2^R(h,\{y\})$. In an analogous way one checks that $\Theta$ is injective.
Now we prove that  $\Theta$ is continuous, for this it is enough to check that $\beta: G\times X  \ni (g,x)\mapsto [g,\{x\}]\in \mathcal{H}_G$, is continuous. For this notice that 
		$\varphi: G\times X\ni (g,x)\mapsto (g,\{x\})\in G\times \mathcal{H}$, is continuous because of Lemma \ref{proph}. Also, $\beta=\pi\circ\varphi$, where $q_{\Hh}: G\times\mathcal{H} \rightarrow \mathcal{H}_G$ is the quotient map, form this  $\beta$ is continuous, and so is  $\Theta.$ Now we need to show that $\Theta^{-1}: {\rm Im}(\Theta)\rightarrow X_G$ is continuous. 
 Let $U\subseteq X_G$ be an open set and $[g_0,\{x_0\}]\in {\rm Im}(\Theta)$  such that $[g_0,x_0]\in U$. Then  $(g_0,x_0)\in q^{-1}(U)$ and there exists open sets  $V\subseteq G$ y $W\subseteq X$ such that $(g_0,x_0)\in V\times W\subseteq q^{-1}(U)$. Take  $Z:= q_\Hh(V\times \langle W\rangle_{\Hh})\cap {\rm Im}(\Theta)$. Since $q_\Hh$ is open, then $Z$ is open in  ${\rm Im}(\Theta)$ and $[g_0,\{x_0\}]\in Z$. On the other hand, take  $[r,\{s\}]\in Z$ we check that $\Theta^{-1}([r,\{s\}])=[r,s]\in U$. For this take  $(v,F)\in V\times \langle W\rangle_{\Hh}$ such that  $[v,F]=q_\Hh(v,F)=[r,\{s\}]$. Then  $F=\{w\}$ for some $w\in W$ and \begin{center}
			$\Theta^{-1}([r,\{s\}])=\Theta^{-1}([v,\{w\}])=[v,w]=q(v,w)\in q(V\times W)\subseteq U$,
	\end{center}
 this shows that $\Theta^{-1}$ is continuous and  $\Theta$ is an embedding. 
\end{proof}

\begin{lem}\label{l2}
	Let  $\theta$ be a partial action on $X$ and  $2^{\theta}$ be the induced  partial action of $G$ on $\Hh$, then $2^R$ is closed in  $(G\times \mathcal{H})^{2}$ provided that $R$ is closed in $(G\times X)^{2}.$
\end{lem}
\begin{proof}
	Take $((g,A),(h,B))\in (G\times \mathcal{H})^2\setminus 2^R$, we have two cases to consider:
	 
\noindent	{\bf Case 1:  $A\notin \langle X_{g^{-1}h}\rangle_{\Hh}$.} Then there exists  $a\in A \cap (X\setminus X_{g^{-1}h}),$ and  $((g,a),(h,b))\in (G\times X)^2\setminus R$, for any  $b\in B$. Since  $R$ is  closed  there are open sets  $U_b,Y_b\subseteq G$ and  $V_b,Z_b\subseteq X$ such that  
$$((g,a),(h,b))\in (U_b\times V_b)\times (Y_b\times Z_b)\subseteq (G\times X)^2\setminus R,$$ for any $b\in B.$
The fact that  $B$ is compact implies that there are $b_1,\cdots,b_n\in B$ for which $B\subseteq \bigcup\limits_{i=1}^n Z_{b_i}$.  Write $U:=\bigcap\limits_{i=1}^nU_{b_i}$, $V:=\bigcap\limits_{i=1}^nV_{b_i}$ and  $Y=\bigcap\limits_{i=1}^nY_{b_i}$.  Then $A\in\langle X,V\rangle_{\Hh}$ and $((g,A),(h,B))\in (U\times \langle X,V\rangle_{\Hh})\times(Y\times \langle Z_{b_1},\cdots,Z_{b_n}\rangle_{\Hh}). $ Now we show that \begin{equation*}\label{closed}
		 (U\times \langle X,V\rangle_{\Hh})\times(Y\times \langle Z_{b_1},\cdots,Z_{b_n}\rangle_{\Hh})\subseteq (G\times \mathcal{H})^2\setminus 2^R.
	\end{equation*}
	For this take $((r,C),(s,D))\in (U\times \langle X,V\rangle_{\Hh})\times(Y\times \langle Z_{b_1},\cdots,Z_{b_n}\rangle_{\Hh})$. For $c\in C\cap V$ and $d\in D$, there is  $1\leq j\leq n$ such that  $d\in Z_{b_j}$, then $((r,c),(s,d))\in (U_{b_j}\times V_{b_j})\times(Y_{b_j}\times Z_{b_j})\subseteq (G\times X)^2\setminus R$ which implies  $c\notin X_{r^{-1}s}$ or  $c\in X_{r^{-1}s}$ and  $\theta_{s^{-1}r}(c)\neq d$. If  $c\notin X_{r^{-1}s}$, then  $C\notin \langle X_{r^{-1}s}\rangle_{\Hh}$  and we have done.  Now suppose  $c\in X_{r^{-1}s}$ and $\theta_{s^{-1}r}(c)\neq d$. If $\theta_{s^{-1}r}(c)\in D$, by a similar argument as above we  get $((r,c),(s,\theta_{s^{-1}r}(c))\notin R$, which leads to a contradiction. Then, $\theta_{s^{-1}r}(C)\neq D$ and $((r,C),(s,D))\notin 2^R$.

\noindent	{\bf Case 2. }$A\subseteq X_{g^{-1}h}.$  Then $\theta_{h^{-1}g}(A)\neq B$. Suppose that there exists  $a\in A$ such that $\theta_{h^{-1}g}(a)\notin B$. Then  $((g,a),(h,b))\notin R$, for any $b\in B,$ we  argue as in Case 1 to obtain   $b_1,\cdots,b_n\in B$ and  families $\{U_{b_i}\}_{i=1}^n$, $\{Y_{b_i}\}_{i=1}^n$ of open subsets of  $G$ such that  $g\in U:=\bigcap\limits_{i=1}^nU_{b_i}$ and $h\in Y:=\bigcap\limits_{i=1}^nY_{b_i}$. Also there are families $\{V_{b_i}\}_{i=1}^n$ and $\{Z_{b_i}\}_{i=1}^n$ of open subsets of  $X$ such that  $a\in V:=\bigcap\limits_{i=1}V_{b_i}, B\in \langle Z_{b_1},\cdots, Z_{b_n}\rangle_{\Hh}$ and  $(U_{b_i}\times V_{b_i})\times(Y_{b_i}\times Z_{b_i})\subseteq (G\times X)^2\setminus R$, for any  $i=1,\cdots n$. As in Case 1 we get
	\begin{center}
		$((g,A),(h,B))\in (U\times \langle X,V\rangle_{\Hh})\times(Y\times \langle Z_{b_1},\cdots,Z_{b_n}\rangle_{\Hh})\subseteq (G\times \mathcal{H})^2\setminus 2^R.$
	\end{center}
	To finish the proof, suppose that there is  $b\in B$ such that  $\theta_{h^{-1}g}(a)\neq b$, for each  $a\in A$. If $a\in A,$ then  $((g,a),(h,b))\notin R$ and   there are open sets   $U_a,Y_a\subseteq G$  and  $V_a,Z_a\subseteq X$ such that $((g,a),(h,b))\in (U_a\times V_a)\times (Y_a\times Z_a)\subseteq (G\times X)^2\setminus R$. The compactness of  $A$  implies that there are  $a_1,\cdots,a_n\in A$ such that $A\subseteq \bigcup\limits_{i=1}^nV_{a_i}$. Write $U:=\bigcap\limits_{i=1}^nU_{a_i}$, $Y':=\bigcap\limits_{i=1}^nY_{a_i}$ and  $Z:=\bigcap\limits_{i=1}^nZ_{a_i}$. Now \begin{center}
		$((g,A),(h,B))\in (U\times\langle V_{a_1},\cdots,V_{a_n}\rangle_{\Hh})\times (Y'\times \langle X,Z\rangle_{\Hh})\subseteq (G\times \mathcal{H})^2\setminus 2^R.$
	\end{center}
Indeed,  let $((r,C),(s,D))\in (U\times\langle V_{a_1},\cdots,V_{a_n}\rangle_{\Hh})\times (Y'\times \langle X,Z\rangle_{\Hh})$ and  $d\in D\cap Z$. For $c\in C$,  there is  $1\leq j\leq n$ such that $c\in V_{a_j}$ therefore $((r,c),(s,d))\in (U_{a_j}\times V_{a_j})\times (Y_{a_j}\times Z_{a_j})\subseteq (G\times X)^2\setminus R$. Moreover, $((s,d),(r,c))\notin R$ and  $d\notin X_{s^{-1}r}$ or $d\in X_{s^{-1}r}$ and $\theta_{r^{-1}c}(d)\neq c$.  In the case $d\notin X_{s^{-1}r}$, we obtain  $D\notin\langle X_{s^{-1}r} \rangle_{\Hh}$ and  $((r,C),(s,D))\notin 2^R$. Thus it only remains to consider the case  $d\in X_{s^{-1}r}$ and  $\theta_{r^{-1}c}(d)\neq c$. If  $\theta_{r^{-1}s}(d)\in C$, as above  we get $((s,d),(r,\theta_{r^{-1}s}(d)))\notin R$, which leads to a contradiction. This shows   $\theta_{r^{-1}s}(d)\notin C$, and $((r,C),(s,D))\notin 2^R$.
\end{proof}
Combining \cite[Lemma 34]{PU} with Lemma \ref{l1},  Lemma \ref{l2} and using that the quotient map to the globalization is open we obtain the following. 
\begin{teo}\label{equivh}
	Let $\theta$  be a partial action of $G$ on $X$. Then $X_G$ is Hausdorff if and only if $\mathcal{H}_G$ is Hausdorff.
\end{teo}

Recall that a locally  compact Cantor space is a  locally  compact Hausdorff space with a countable basis of clopen sets and no isolated points. 

We proceed with the next.
\begin{pro} Let $X$ be a  metric compact Cantor space, $G$ a countable discrete group and suppose that $\theta=(\theta_g, X_g)_{g\in G}$ is a partial action of $G$ on $X$ such that $X_g$ is clopen for all $g\in G.$ Then $(2^{X})_G$ is a locally compact Cantor space.
\end{pro}
\begin{proof} Since $X$ is a compact Hausdorff space, then $2^X$ is a compact Hausdorff space. Moreover since $X$ is metric, we get from  \cite[Proposition 8.4]{NW} that  $\Hh_1=2^X$ is a  Cantor space. Also $\langle X_g \rangle$ is clopen for all $g\in G$ and the result follows from \cite[Proposition 3.3]{EGG}.
\end{proof}

Now we shall work with the hyperspace $\Hh_3=F(X)$ consisting of finite subsets of $X.$ The following result shows that the enveloping space $(\Hh_3)_G$ is $T_1,$ provided that $X_G$ is.
\begin{pro}\label{fin}
	Let  $\theta$ be a topological partial action of  $G$ on $X$ and  $2^{\theta}$ be the induced partial action of $G$ on $\Hh_3$. If $X_G$ is $T_1,$ then  $(\Hh_3)_G$ is $T_1$.
\end{pro}
\begin{proof}
	Let  $A=\{a_1,\cdots,a_n\}\in \Hh_3$ and  $g\in G$,  and  $q_{\Hh_3}: G\times \Hh_3\rightarrow (\Hh_3)_G$  be the corresponding quotient map. We need to show that  \begin{center}
		$\pi^{-1}([g,A])=\{(h,F)\in G\times \Hh_3: \exists (g^{-1}h)\cdot F\ {\rm and}\ (g^{-1}h)\cdot F=A\}$
	\end{center} is closed in $ G\times \Hh_3.$ Take  $(h,F)\notin \pi^{-1}([g,A])$. There are two cases to consider. 

{\bf Case 1: }$\nexists (g^{-1}h)\cdot F$. Then there is $f\in F$ such that  $\nexists(g^{-1}h)\cdot f$. Since  $X_G$ is $T_1$,  for  $1\leq i\leq n$ there are open sets  $U_i\subseteq G$ y $V_i\subseteq X$ for which  \begin{center}
			$(h,f)\in U_i\times V_i\subseteq (G\times X)\setminus q^{-1}([g,a_i])$.
		\end{center}
		Take  $U:=\bigcap_{i=1}^{n} U_i$ y $V=\bigcap_{i=1}^{n} V_i$. Note that  $(h,F)\in U\times \langle X,V\rangle_{\Hh_3}\subseteq (G\times \Hh_3)\setminus q_{\Hh_3}^{-1}:([g,A])$. Indeed, if  $(t,B)\in U\times \langle X,V\rangle_{\Hh_3}$ and  $b\in B\cap V$ we have  $(t,b)\notin q^{-1}([g,a_i])$ for any  $1\leq i\leq n$. If $\nexists g^{-1}t\cdot b$, then  $\nexists (g^{-1}t)\cdot B$ and  $(t,B)\notin  q_{\Hh_3}^{-1}([g,A])$. On the other hand, if $\exists (g^{-1}t)\cdot b$, then $(g^{-1}t)\cdot b\neq a_i$, for each $1\leq i\leq n$, then  $(g^{-1}t)\cdot b\notin A$ and $(t,B)\notin  q_{\Hh_3}^{-1}([g,A])$.

{\bf Case 2: }$\exists (g^{-1}h)\cdot F$ and $(g^{-1}h)\cdot F\neq A$. If  there is   $f\in F$ for which  $(g^{-1}h)\cdot f\notin A$ we get  $(h,f)\notin q^{-1}([g,a_i])$ for  $1\leq i\leq n$ and we proceed as in Case 1. If there is  $a\in A$ such that $(g^{-1}h)\cdot f\neq a$, for any  $f\in F$ write $F=\{f_1,\cdots,f_k\}$,  then $(h,f_j)\notin q^{-1}([g,a])$ for each  $1\leq j\leq k$. Hence there are open sets  $U\subseteq G$ and  $V\subseteq X$ such that  $(h,f_j)\in U\times V\subseteq (G\times X)\setminus q^{-1}([g,a])$, for every $1\leq j\leq k$. Note that  $(h,F)\in U\times \langle V\rangle_{\Hh_3}\subseteq (G\times F(X))\setminus  q_{\Hh_3}^{-1}([g,A])$. Indeed, if  $(t,B)\in U\times \langle V\rangle_{\Hh_3}$. If $\nexists (g^{-1}t)\cdot B$, then  $(t,B)\notin  q_{\Hh_3}^{-1}([g,A])$. In the case  $\exists g^{-1}t\cdot B$, we get that for any $b\in B$ the pair $(t,b)$ belongs to $ U\times V$ and thus  $(g^{-1}t)\cdot b\neq a$ which gives $(t,B)\notin  q_{\Hh_3}^{-1}([g,A]),$ as desired.
\end{proof}
Combining  Lemma \ref{l1} and Proposition \ref{fin} we get.
\begin{coro} Let  $\theta$ be a topological partial action of  $G$ on $X$ and  $2^{\theta}$ be the induced partial action of $G$ on $\Hh_3$. Then  $(\Hh_3)_G$ is $T_1$ if and only if $X_G$ is $T_1.$
\end{coro}

We proceed with the next

\begin{coro}\label{separ2}
	Let $G$  be a separable group and $\theta$  be a continuous  partial action of $G$ on  $X$ such that $G*X$ is open and $X$  is separable. Take $\Hh\in \{\Hh_1,\Hh_3\}$ then the following assertions hold.
	\begin{enumerate}
		\item [(i)] $X_G$ is separable.
		\item [(ii)] $\Hh_G$ is separable, where  
		\item [(iii)] If  $X_G$ is $T_1$,  then  $\Hh({X_G})$ and   $\Hh({F(X)_G})$ are separable.
\end{enumerate}
\end{coro}
\begin{proof}
	 (i) Since $\theta$ is continuous with open domain then $\iota(X)$ is open in $X_G$ and  \eqref{iota} is a homeomorphism onto $\iota(X),$ in particular $\iota(X)$ is separable, moreover the map $\mu$ given in \eqref{action} acts continuously in $X_G$ and  $G\cdot \iota(X)=X_G$ thus the result follows by (i) in Lemma \eqref{separ}.

 (ii) Since $X$ is separable and  $T_1$, then  $\Hh$ is separable. Then by  $(i)$ and $(ii)$ of Theorem \ref{teo2.8} and  Lemma \eqref{separ} we get that  $\Hh_G$ is separable.
	
(iii)	By $(i)$ the space $X_G$ is separable and follows that $\Hh({X_G})$ is separable. Finally, by   Proposition \ref{fin}  we have that  $F(X)_G$ is $T_1$, moreover $F(X)_G$ is separable thanks to (ii), and thus  $\Hh({F(X)_G})$ is separable.
\end{proof}

\begin{exe}\cite[Example 4.8]{PU} Consider the partial action of $\mathbb{Z}$ on $X=[0,1]$  given by $\theta_0={\rm id}_X$  and $\theta_n={\rm id}_{[0,1)}, n\neq 0$ then  $\theta$ is continuous with open domain and $X_\mathbb{Z}$ is  $T_1.$ Thus by Corollary \ref{separ2}  the spaces $ X_\mathbb{Z}, \Hh_\mathbb{Z}, \Hh({[0,1]_\mathbb{Z}})$ and   $\Hh({F([0,1])_G})$ are separable, where $\Hh\in \{\Hh_1,\Hh_3\}.$
\end{exe}
Now we shall deal with the regularity condition. 
\begin{teo}\label{regu}
	Let  $\theta: G*X\rightarrow X$ be a continuous partial action with closed domain. Then the spaces $X_G$ and $\Hh_G$ are regular.
\end{teo}
\begin{proof} Let $\iota$ be the embedding map defined in \eqref{iota} then $G\iota(X)=X_G,$ we shall prove that $\iota(X)$ is closed and regular. Let $q: G\times X\to X_G$ the quotient map, then $q\m(\iota(X))=G*X$ is closed in $G\times X$ which shows that $\iota(X)$ is closed in $X_G.$ Now since $X$ is a compact Hausdorff space we have that  $\iota(X)$ is regular  and thus  $X_G$ is regular  thanks to item (ii) of Lemma \ref{separ}. On the other hand,  we have that  $\Hh$ es compact and Hausdorff, $2^{\theta}$ is continuous ((ii) of Theorem {2.8}) and  $G*\Hh$ is closed thanks to Proposition \ref{cer}, then it is enough to apply (ii) of Lemma \ref{separ}.  
\end{proof}

\begin{rem} In \cite[Theorem 4.6]{PU} are presented other conditions for the space $X_G$ being regular.
\end{rem}
\section{On the category  \textbf{$G\curvearrowright$ CH} }\label{cate}
We shall use some of the above results to construct a monad in the category of partial actions on compact Hausdorff spaces. First recall the next.

\begin{defi}\label{mor}
	Let  $\phi=(\phi_g, X_g)_{g\in G}$ and  $\psi=(\psi_g, Y_g)_{g\in G},$  be partial actions of $G$ on the spaces $X$ and $Y$, respectively. A $G$-map $f:\phi \to \psi$ is a  continuous function   $f:X\to Y$ such that:
	\begin{enumerate}
		\item [(i)] $f(X_g)\subseteq Y_g,$
		\item [(ii)] $f(\phi_g(x))=\psi_g(f(x)),$ for each $x\in X_{g^{-1}},$
	\end{enumerate}
	for any $g\in G.$ 	 If moreover $f$ is a homeomorphism and $f^{-1}$ is $G$-map, we say that $\phi$ are  $\psi$  equivalent. 
\end{defi}
 We denote by  \textbf{$G\curvearrowright$ Top} the category whose objects are topological  partial actions of $G$ on topological spaces and morphisms are $G$-maps defined as above.  Also, we denote by \textbf{$G\curvearrowright$ CH}  the subcategory of \textbf{$G\curvearrowright$ Top} whose objects are topological partial actions of $G$ on compact Hausdorff spaces. It follows by  Theorem \ref{teo2.8} that there is a functor $2^-: G\curvearrowright \mathbf{CH}\rightarrow G\curvearrowright \mathbf{CH}$.

\subsection{The  monad $\mathbb{I}$}
Recall the next.
\begin{defi}
	Let $\mathcal{C}$  be a category. A monad in $\mathcal{C}$ is a triple $(T,\eta,\mu)$, where  $T:\mathcal{C}\to\mathcal{C}$ is an endofunctor,  $\eta: Id_{\mathcal{C}}\Longrightarrow T$ and $\mu: T^2\Longrightarrow T$ are natural transformations  such that: \begin{equation}\label{mon}
		\mu\circ T\eta=\mu\circ \eta T=1_T\,\,\,\text{ and }\,\,\,\mu\circ\mu T=\mu\circ T\mu.
	\end{equation}
	
\end{defi}

Given an  object $\alpha=(\alpha_g, X_g)_{g\in G}\in G\curvearrowright \mathbf{CH}.$ We have by Lemma \ref{proph} that the map  $\eta_{\alpha}: X\ni x\mapsto \{x\}\in 2^X$,  is a continuous function. From this it is not difficult to see that  $\eta_{\alpha}: \alpha\to 2^\alpha$ is a a morphism in $G\curvearrowright \mathbf{CH}$. Moreover,  for an object  $\beta=(Y_g , \beta_g)_{g\in G}$ in   $G\curvearrowright \mathbf{CH}$ and a morphism $f:\alpha\to \beta $ the diagram:
	\begin{center}
		$\xymatrix{X \ar[d]_-{f} \ar[r]^-{\eta_{\alpha}}& 2^X\ar[d]^-{2^{f}}\\
			Y\ar[r]_-{\eta_{\beta}}& 2^Y
		}$
		\end{center}
is commutative.  Thus the family $\eta=\{\eta_{\alpha}\}_{\alpha\in G\curvearrowright \mathbf{CMet}}:Id_{G\curvearrowright \mathbf{CMet}}\Longrightarrow 2^-$ is a natural transformation. Now set  $\mu_{\alpha}: 2^{2^X}\ni A\mapsto \cup A \in 2^X$, by Lemma \ref{proph} $\mu_{\alpha}$  is continuous. We shall check that  $\mu_{\alpha}:2^{2^{\alpha}}\to 2^{\alpha}$ is a morphism in  $G\curvearrowright \mathbf{CH}$ . 
	\begin{enumerate}
		\item [(i)] Take $g\in G$ and  $A\in \langle\langle X_g\rangle\rangle$. Then $A\subseteq \langle X_g\rangle$ and  $\mu_{\alpha}(A)=\cup A\subseteq X_g$, that is  $\mu_{\alpha}(A)\in \langle X_g\rangle$.
		\item [(ii)] For  $A\in \langle\langle X_{g^{-1}}\rangle\rangle$ we have  $2^{\alpha_g}[A]=\{\alpha_g(F): F\in A\}$, then 
			$\mu_{\alpha}(2^{2^{\alpha_g}}(A))=\mu_{\alpha}(2^{\alpha_g}[A])=\cup 2^{\alpha_g}[A]=\alpha_g(\cup A)=2^{\alpha_g}(\cup A)=2^{\alpha_g}(\mu_{\alpha}(A))$,
as desired.
\end{enumerate}
Now we prove that  $\mu=\{\mu_{\alpha}\}_{\{\alpha\in G\curvearrowright \mathbf{CH}\}}: (2^-)^2\Longrightarrow 2^-$ is a  natural transformation. For this take $\beta=(Y_g , \beta_g)_{g\in G}$ in   $G\curvearrowright \mathbf{CH}$ and a morphism $f:\alpha\to \beta $ in $G\curvearrowright \mathbf{CH}$. Consider the diagram
		\begin{equation}\label{d1}\xymatrix{2^{2^X} \ar[d]_-{2^{2^f}} \ar[r]^-{\mu_{\alpha}}& 2^X\ar[d]^-{2^{f}}\\
			2^{2^Y}\ar[r]_-{\mu_{\beta}}& 2^Y
		}\end{equation}
		
Let $A\in 2^{2^X}$, then  $2^f[A]=\{f(B): B\in A\}$ and  
		$2^f(\mu_{\alpha}(A))=2^f(\cup A)=f(\cup A)=\cup 2^f[A]=\mu_{\beta}(2^f[A])$
thus the diagram \eqref{d1}  is commutative.
\begin{teo} \label{monad}Let $\eta$ and $\mu$ be as above. Then the triple
	 $\mathbb{I}=(2^-,\eta,\mu)$ forms a monad in the category $G\curvearrowright \mathbf{CH}$.
\end{teo}
\begin{proof} It remains to prove that equalities in \eqref{mon} hold. Let 
	Sea $\alpha$ be an object in  $G\curvearrowright \mathbf{CH}$ 
Since$(\eta2^-)_{\alpha}=\eta_{2^{\alpha}}$, we have that  $\mu_{\alpha}\circ \eta_{2^{\alpha}}: 2^X\ni A\mapsto A\in 2^X,$ which gives$ \mu\circ \eta2^-=1_{2^-}$. Also, $(2^-\eta)_{\alpha}=2^-(\eta_{\alpha})=2^{\eta_{\alpha}}$, and  $\mu_{\alpha}\circ 2^{\eta_{\alpha}}: 2^X\ni A\mapsto A\in 2^X,$ which shows $\mu\circ 2^-\eta-=\mu\circ \eta2^-=1_{2^-}$. Finally, since  $(\mu2^-)_{\alpha}=\mu_{2^{\alpha}}$ y $(2^-\mu)_{\alpha}=2^-(\mu_{\alpha})$, we have \begin{center}
		$(\mu\circ \mu2^-)_{\alpha}=\mu_{\alpha}\circ \mu_{2^{\alpha}}=\mu_{\alpha}\circ2^{\mu_{\alpha}}=(\mu\circ2^-\mu)_{\alpha}$,
	\end{center}
	lthus $\mu\circ \mu2^-=\mu\circ 2^-\mu$  and  $(2^-,\eta,\mu)$ is a monad.
\end{proof}

	\end{document}